\documentclass[12pt]{amsart}

\usepackage[foot]{amsaddr}

\usepackage{euscript,amsfonts,amssymb,amsmath,amscd,amsthm,enumerate,hyperref}
\usepackage{mathrsfs}
\usepackage[margin=1.1in]{geometry}
\usepackage{bbm}
\usepackage{enumitem}
\usepackage[giveninits=true,style=alphabetic,backend=biber,maxnames=9,maxalphanames=9,doi=true,isbn=true,url=false,eprint=false]{biblatex}
\renewbibmacro{in:}{}
\AtEveryBibitem{
  \clearlist{language}
}
\newcommand{\lb}{\left}
\newcommand{\rb}{\right}

\newcommand{\R}{\mathbb{R}}

\newcommand{\ind}{\mathbbm{1}}

\newcommand{\wt}{\widetilde}
\newcommand{\wh}{\widehat}
\newcommand{\ubar}{\underline}
\newcommand{\veps}{\varepsilon}

\newcommand{\bP}{\mathbb{P}}

\newcommand{\E}{\mathbb{E}}

\theoremstyle{plain}
\newtheorem{thm}{Theorem}[section]
\newtheorem{prop}[thm]{Proposition}
\newtheorem{coro}[thm]{Corollary}
\newtheorem{lem}[thm]{Lemma}

\newtheorem*{conj}{Conjecture}

\theoremstyle{definition}
\newtheorem{defn}[thm]{Definition}

\newtheorem{rmk}[thm]{Remark}

\title[Sensitivity to Initial Data for the Supercooled Stefan Problem]{Sensitivity to Initial Data for Physical and Minimal Solutions of the Supercooled Stefan Problem}

\author{Graeme Baker}
\date{\today}
\address{Department of Statistics, Columbia University, New York, NY 10027, USA.}
\email{g.baker@columbia.edu}

\begin{document}

\begin{abstract}
We address the problem of well-posedness for physical and minimal solutions to a probabilistic reformulation of the supercooled Stefan problem by investigating the sensitivity of these solutions to changes in the initial data.
We show that the solution map for physical solutions is continuous under perturbations to the initial condition $X_{0-}$ (in the weak sense of probability measures), provided that $X_{0-}$ admits a unique physical solution.
Furthermore, we show continuous dependency of the solution map for minimal solutions when the data is shifted to the right; however, continuity of this solution map is shown to fail at $X_{0-}$ for shifts to the left unless uniqueness of physical solutions holds.
As a result, we show that the question of whether the solution map for minimal solutions is continuous at $X_{0-}$ is equivalent to the question of whether the given data $X_{0-}$ admits a unique physical solution.
\end{abstract}

\maketitle

\section{Introduction}

We consider the McKean--Vlasov equation
\begin{equation}\label{eq:McK-V}
\left\{
\begin{aligned}
    X_t&=X_{0-}+B_t-\Lambda_t,\quad t\ge 0\\
    \tau &= \inf \{t\ge 0: X_t\le 0\},\\
    \Lambda_t&=\alpha\bP(\tau\le t),\quad t\ge 0,
\end{aligned}
\right.
\end{equation}
where $\alpha>0$, the initial condition $X_{0-}$ is a random variable on $\R$, and $B$ is a Brownian motion independent of $X_{0-}$. In this article, we study \emph{physical} and \emph{minimal} solutions to this problem, and the sensitivity of these solutions to perturbations in the initial data. Throughout the work, we let $D(I)$ denote the space of c\`adl\`ag functions from a given interval $I$ to $\R$. \emph{A priori}, we are concerned with solutions $\Lambda\in D([0,\infty))$ to \eqref{eq:McK-V} which are in the set
\begin{align}
M=\{\ell\in D([0,\infty)):\ell_{0-}=0,\text{ and }\ell\text{ is increasing}\}.
\end{align}

The system \eqref{eq:McK-V} arises as a probabilistic reformulation of the one-phase supercooled Stefan problem in one dimension: a free boundary problem for the heat equation which models the solidification of a supercooled liquid. In that setting, the function $t\mapsto\Lambda_t$ describes the time evolution of the boundary between liquid and solid \cite{delarue_global_2022}. For general initial conditions, the corresponding free boundary partial differential equation (PDE) may exhibit finite-time blowup of the derivative of the boundary, and admit multiple weak solutions (in the sense of PDE). Equation \eqref{eq:McK-V} has also been used as a model for systemic risk in interconnected banking systems, where $\tau$ denotes the time of default for a representative bank whose assets at time $t$ are given by $X_t$, and $\Lambda_t$ is the proportion of banks which have defaulted by time $t$ (see \cite{nadtochiy_particle_2019}, as well as extensions such as \cite{nadtochiy_mean_2019,hambly_spde_2019,baker_singular_2023}, among others).
In the case of either application, equation \eqref{eq:McK-V} can be seen as the mean-field limit of the interacting particle system
\begin{equation}\label{eq:finite}
\left\{
\begin{aligned}
    X_t^{i}&=X_{0-}^i+B_t^i-L_t^N,\quad t\ge 0\\
    \tau^i &= \inf \{t\ge 0: X_t^i\le 0\},\\
    L_t^N&=\frac{\alpha }{N}\sum_{i=1}^N \ind_{\{\tau^i\le t\}} ,\quad t\ge 0,
\end{aligned}
\right.
\end{equation}
where particles $(X^i)_{1\le i\le N}$ have independent and identically distributed initial conditions $X_{0-}^i\stackrel{d}{=}X_{0-}$ and $(B^i)_{1\le i\le N}$ is a family of independent Brownian motions. 

\medskip

Depending on the initial distribution $X_{0-}$, both \eqref{eq:McK-V} and \eqref{eq:finite} may admit multiple solutions (see, for instance, \cite[Subsections 2.1 and 3.1]{delarue_particle_2015}). There is a natural way of resolving jump cascades which occur in \eqref{eq:finite} when the hitting time of one particle triggers one or more other particles to hit (for specifics, we refer the reader to the detailed discussion in \cite[Section 3]{delarue_particle_2015}). The condition used to resolve the cascades in the finite problem \eqref{eq:finite} informs which jump sizes are admissible in the limiting problem and gives rise to the following definition.

\begin{defn}
A solution $\Lambda\in M$ of \eqref{eq:McK-V} is deemed \emph{physical} provided that
\begin{align}\label{eq:phys}
\Lambda_{t}-\Lambda_{t-}=\inf\{x>0:\alpha\bP(\tau\ge t, X_{t-}\in(0,x])<x\}
\end{align}
is satisfied at all times $t\ge 0$.
\end{defn}

The physical jump condition was introduced in \cite{delarue_particle_2015} in the context of mean field networks of integrate-and-fire neurons, and later for the supercooled Stefan problem in \cite{delarue_global_2022}.
For the purpose of applications, the discontinuities of $\Lambda$ represent finite-time blowup for the free boundary in the supercooled Stefan problem, or an event where a macroscopic proportion of banks simultaneously default in the context of systemic risk. The physical jump condition \eqref{eq:phys} ensures that when a discontinuity of $\Lambda$ occurs, it is the smallest possible choice of jump among c\`adl\`ag solutions \cite[Proposition 1.2]{hambly_mckean--vlasov_2018}.

\medskip

Provided that some assumptions are made on the regularity of the initial condition $X_{0-}$, equation \eqref{eq:McK-V} is known to admit a unique physical solution. For instance, when $X_{0-}$ possesses a bounded density on $[0,\infty)$ which changes monotonicity finitely often, then the physical solution is unique and the regularity of it is fully characterized by \cite[Theorems 1.1 and 1.4]{delarue_global_2022}. Uniqueness of the physical solutions has also been established using a weaker assumption in \cite{mustapha_well-posedness_2023}, which allows for oscillatory initial densities such as those which behave like $(1+\sin 1/x)/2$ near the origin. As of writing, we do not know of an initial condition which gives rise to multiple physical solutions. Of course, we do know that such an initial condition would have to be irregular enough to invalidate the assumptions of \cite{delarue_global_2022,mustapha_well-posedness_2023}.

\medskip

With regards to the supercooled Stefan problem as a physical model, the physical solutions of \cite{delarue_global_2022} give a notion of global-in-time solutions in the presence of possible blow-ups, and these solutions are unique under the aforementioned technical assumptions.
A mathematical model for a physical phenomenon is well-posed provided that (i) there exists (appropriately defined) solutions, (ii) these solutions are unique, and (iii) they depend continuously on initial data \cite[page 7]{evans_partial_2010}. 
In Section \ref{sec:phys}, we show that the following holds
\begin{thm}
\label{thm:phys}
Let $\Lambda\in M$ solve \eqref{eq:McK-V} with initial data $X_{0-}$ such that $\E|X_{0-}|<\infty$. Let $(\Lambda^n)_{n\ge 1}$ be a sequence of physical solutions to \eqref{eq:McK-V}, where the corresponding initial conditions $(X_{0-}^n)_{n\ge 1}$ all have finite expectation. Suppose that $X_{0-}^n\to X_{0-}$ weakly and $\Lambda^n\to \Lambda$ in $(D([-1,\infty)),\operatorname{M1})$. Then $\Lambda$ satisfies
the physical jump condition
\begin{align}
\Lambda_{t}-\Lambda_{t-}=\inf\{x>0:\alpha\bP(\tau\ge t, X_{t-}\in(0,x])<x\}
\end{align}
at all $t\ge 0$. 
\end{thm}

\begin{rmk}\label{rmk:cheaptrick}
The convergence is stated in $(D([-1,\infty)),\operatorname{M1})$, where $\operatorname{M1}$ denotes the Skorokhod $\operatorname{M1}$ topology. We work on $(D([-1,\infty)),\operatorname{M1})$ rather than $(D([0,\infty)),\operatorname{M1})$ since the $\operatorname{M1}$ topology enforces continuity at the left endpoint. Throughout this work, we will use the embedding $\iota :D([0,\infty))\hookrightarrow D([-1,\infty))$ defined by $\iota f(x)=f(0-)$ for $x\in [-1,0)$ whenever $f\in D([0,\infty))$ (we will omit the symbol $\iota$). This gives us convergence at the left endpoint for free (c.f.~\cite[Remark 4.3]{cuchiero_propagation_2020}). 
\end{rmk}

With Theorem \ref{thm:phys}, we see that continuous dependence holds for the physical solution map when the initial condition $X_{0-}$ admits a unique physical solution. Therefore, we see that the question of well-posedness of physical solutions to \eqref{eq:McK-V} reduces to the question \emph{for which initial conditions does \eqref{eq:McK-V} admit a unique physical solution?} In this work, we highlight the centrality of this question for both the well-posedness of physical solutions and the well-posedness of minimal solutions, which we define next.

\medskip

Another notion of solution to equation \eqref{eq:McK-V}, known as the minimal solution and introduced in \cite{cuchiero_propagation_2020}, is arrived at as follows: first, let $\ell\in M$ be a candidate for the free boundary. Consider the problem
\begin{equation}\label{eq:Gammadef}
\left\{
\begin{aligned}
X^{\ell}_t=&X_{0-}+B_t- \ell_t,\quad t\ge 0\\
\tau^{\ell}=&\inf \{t\ge 0:X^{\ell}_t\le 0\}\\
\Gamma[\ell;X_{0-}]_t:=&\alpha\bP(\tau^{\ell}\le t), \quad t\ge 0.
\end{aligned}
\right.
\end{equation}
We see that the hunt for solutions to \eqref{eq:McK-V} amounts to finding fixed points of $\Gamma[~\cdot~;X_{0-}]$.
The operator $\Gamma[~\cdot~;X_{0-}]$ is continuous on $M$ by \cite[Proposition 2.1]{cuchiero_propagation_2020}, and it is easy to see that it is monotone in the sense that 
\begin{align}
\ell^1_t\le\ell^2_t\text{ for all }t\ge 0\implies \Gamma[\ell^1;X_{0-}]_t\le \Gamma[\ell^2;X_{0-}]_t\text{ for all }t\ge 0.
\end{align}

\medskip

For any solution $\Lambda$ of \eqref{eq:McK-V}, we must have $0\le \Lambda_t$ for all $t\ge 0$, as $\frac{1}{\alpha}\Lambda_t$ is the probability of some event. Hence
\begin{align}
\Gamma[0;X_{0-}]_t\le \Gamma[\Lambda;X_{0-}]_t=\Lambda_t,\quad t\ge 0,
\end{align}
since $\Lambda$ is a fixed point of $\Gamma[~\cdot~;X_{0-}]$. Define $\Gamma^n[~\cdot~;X_{0-}]$ recursively by
\begin{align}
\Gamma^n[~\cdot~;X_{0-}]:=\Gamma^{n-1}[\Gamma[~\cdot~;X_{0-}];X_{0-}],
\end{align}
for $n\ge 2$. Applying $\Gamma^n$ to the inequality $0\le \Lambda_t$ for all $t\ge 0$ yields 
\begin{align}\label{eq:gammaineq}
\Gamma^n[0;X_{0-}]_t\le \Gamma^n[\Lambda;X_{0-}]_t=\Lambda_t,\quad t\ge 0.
\end{align}
In \cite[Proposition 2.3]{cuchiero_propagation_2020}, it is shown that the limit $\ubar{\Lambda}:=\lim_{n\to \infty}\Gamma^n[0;X_{0-}]$ is a solution to \eqref{eq:McK-V}.
\begin{defn}
We call $\ubar{\Lambda}:=\lim_{n\to \infty}\Gamma^n[0;X_{0-}]$ the \emph{minimal} solution to \eqref{eq:McK-V}.
\end{defn}
Taking the limit of \eqref{eq:gammaineq}, we see that $\ubar{\Lambda}_t\le\Lambda_t$, for all $t\ge 0$ when $\Lambda$ is any other solution to \eqref{eq:McK-V}.
The minimality of $\ubar{\Lambda}$ among solutions to \eqref{eq:McK-V} can also be taken as the definition of the minimal solution. Either by construction, or by minimality, we have that the minimal solution is unique. We note the following important connection between minimal and physical solutions:
\begin{prop}[Theorem 6.5 in \cite{cuchiero_propagation_2020}]\label{prop:minisphys}
When $\E|X_{0-}|<\infty$, the minimal solution is a physical solution.
\end{prop}
From \cite{cuchiero_propagation_2020}, there is the following conjecture concerning sensitivity to initial conditions for minimal solutions:
\begin{conj}[Conjecture 6.10 in \cite{cuchiero_propagation_2020}]
Let $X_{0-}$ be given with $\E|X_{0-}|<\infty$. For $x\in\R$, let $\ubar{\Lambda}^x$ denote the minimal solution to \eqref{eq:McK-V} with shifted initial condition $X_{0-}+x$. The map 
$x\mapsto \ubar{\Lambda}^x$ is continuous from $\R$ to $(D([-1,\infty)),\operatorname{M1})$.
\end{conj}

In a positive direction, we are able to show that $x\mapsto \ubar{\Lambda}^x$ is c\`adl\`ag in the following sense:

\begin{thm}
Suppose that $\E|X_{0-}|<\infty$. Then $x\mapsto \ubar{\Lambda}^x$ is right-continuous from $\R$ to $(D([-1,\infty)),\operatorname{M1})$. As $x\uparrow 0$, $\ubar{\Lambda}^x$ converges in $(D([-1,\infty)),\operatorname{M1})$ to some $\Lambda^0\in M$, which is a solution to \eqref{eq:McK-V}.
\end{thm}
\begin{proof}
Right continuity is a consequence of our more general convergence result, Theorem \ref{thm:rc}. The proof of that theorem shows that the left limit gives some solution to \eqref{eq:McK-V}. This is formally stated in Corollary \ref{coro:ll}.
\end{proof}
Returning to the notion of well-posedness, the conjecture effectively states that the problem of finding minimal solutions to \eqref{eq:McK-V} is well-posed among initial conditions with $\E|X_{0-}|<\infty$. 
Proposition \ref{prop:minisphys} tells us that the minimal solutions $\ubar{\Lambda}^x$ are physical, and hence by Theorem \ref{thm:phys}, the left limit point $\Lambda^0$ is physical. Therefore, if the initial condition $X_{0-}$ admits a unique physical solution, then $\Lambda^0=\ubar{\Lambda}^0$ and the conjecture holds (this was previously known, see \cite[Theorem 6.6]{cuchiero_propagation_2020}). 
To complicate matters, the following theorem shows that the non-uniqueness of physical solutions to \eqref{eq:McK-V} implies failure of left-continuity for the map $x\mapsto\ubar{\Lambda}^x$. In fact, the left limit $\Lambda^0=\lim_{x\uparrow 0}\ubar{\Lambda}^x$ is a \emph{maximal} physical solution in that it dominates all other physical solutions at all times.

\begin{thm}
\label{thm:nonunique}
Let $X_{0-}$ be an initial condition with $\E|X_{0-}|<\infty$ and let $\Lambda^0=\lim_{x\uparrow 0}\ubar{\Lambda}^x$. If $\Lambda$ is any physical solution to \eqref{eq:McK-V} with initial condition $X_{0-}$ then $\Lambda^0_t\ge \Lambda_t$ for all $t\ge 0$. As a result, if $X_{0-}$ admits at least two physical solutions, then $\Lambda^0\neq\ubar{\Lambda}^0$.
\end{thm}

Therefore, the question of well-posedness for the minimal solutions to \eqref{eq:McK-V}, just as with well-posedness for physical solutions, reduces to the question \emph{for which initial conditions does \eqref{eq:McK-V} admit a unique physical solution?} In other words, the conjecture holds if and only if all initial conditions with finite expectation admit unique physical solutions.

\medskip

Our general convergence result for minimal solutions, stated below, allows for both deterministic and random shifts to the initial condition: we consider a sequence of initial conditions $(X_{0-}^n)_{n\ge 1}$ and assume that the cumulative distribution functions (CDFs) of the initial conditions increase pointwise and converge to the CDF for $X_{0-}$ at all continuity points of the latter. 
Note that pointwise convergence of CDFs at continuity points of the limit is equivalent to weak convergence of probability measures on $\R$, by the Helly-Bray Theorem (see, for instance, \cite[Theorem 11.1.2]{dudley_real_2002}). Alternatively, we may consider $X_{0-}^n=X_{0-}+Y^n$ for a sequence of independent non-negative random variables $(Y^n)_{n\ge 1}$ with $Y^n\to0$ weakly and $\bP(Y^n\le x)\ge \bP(Y^m\le x)$ for all $x>0$ whenever $n>m$ (we note that this partial ordering is known as stochastic dominance).

\begin{thm}
\label{thm:rc}
Let $X_{0-}$ have CDF $F$. Let $(X_{0-}^n)_{n\ge 1}$ be random variables on $[0,\infty)$ with CDFs $(F^n)_{n\ge 0}$, respectively, such that $F(x) \ge F^n(x)\ge F^m(x)$ for all $x\ge 0$ whenever $m<n$, and $F^n(x)\to F(x)$ whenever $x$ is a continuity point of $F$. For $n\ge0$, let $\ubar{\Lambda}^n$ denote the minimal solution to \eqref{eq:McK-V} with initial data $X_{0-}^n$; and let $\ubar{\Lambda}$ denote the minimal solution with initial data $X_{0-}$. Then $\ubar{\Lambda}^n\to \ubar{\Lambda}$ in $(D([-1,\infty)),\operatorname{M1})$
\end{thm}

In \cite[Proposition 5.1]{hambly_mckean-vlasov_2022}, convergence for the initial conditions $X_{0-}^n=X_{0-}+\xi^n$ is considered, where $\xi^n$ is an exponential random variable with rate $n$. Following the reasoning of \cite[Section 3]{hambly_mckean-vlasov_2022}, we see that this in turn is a generalization of our earlier result \cite[Theorem 1.2]{baker_zero_2022}, where it was assumed that the density for $X_{0-}$ is bounded by $\alpha/2$ 
(this bound renders the physical solution unique and continuous by \cite[Theorem 2.2 and last paragraph of Section 2.1]{ledger_uniqueness_2020}).
Our proof of Theorem \ref{thm:rc} fixes an error in the proof of \cite[Proposition 5.1]{hambly_mckean-vlasov_2022} (see Remark \ref{rmk:error}) and further generalizes the result by using a fact about the zeroes of Brownian motion with a bounded variation drift. Specifically, we make use of \cite[proof of Proposition 3.4(i)]{antunovic_isolated_2011} to show that almost surely none of these zeroes are local maxima. This observation has also found recent use in the works \cite{baker_singular_2023,nadtochiy_stefan_2022}.

\medskip

The outline of the remainder of the paper is as follows: in the following subsection we will cover some technical preliminaries. Then, in Sections \ref{sec:rc} and \ref{sec:phys} we give the proofs of our convergence results for minimal and physical solutions, respectively. Finally, in Section \ref{sec:nonunique}, we prove Theorem \ref{thm:nonunique}, thus showing that non-uniqueness of physical solutions implies the ill-posedness of minimal solutions.

\subsection{Mathematical Preliminaries}

We first state a useful characterization for convergence of monotone functions in the M1 topology. 

\begin{prop}[Corollary 12.5.1 in \cite{whitt_stochastic-process_2002}]\label{prop:M1conv}
Let $D([T_0,T_1])$ denote the space of c\`adl\`ag functions from $[T_0,T_1]$ to $\R$ which are left continuous at $T_0$. Suppose that $(f^n)_{n\ge 1}$ is a family of non-decreasing functions in $D([T_0,T_1])$. Then $f^n\to f$ in $(D([T_0,T_1]),\operatorname{M1})$ if and only if $f^n(t)\to f(t)$ for all $t$ in a dense subset of $[T_0,T_1]$ including $T_0$ and $T_1$.
\end{prop}

Following \cite{whitt_stochastic-process_2002}, we say that $f^n\to f$ in $(D([T_0,\infty)),\operatorname{M1})$ if there is a sequence $(T_N)_{N\ge1}$ with $T_N\nearrow\infty$ such that $f^n\to f$ in $(D([T_0,T_1]),\operatorname{M1})$ for all $N\ge 1$. Note that this imposes convergence at the left endpoint: $f^n(T_0)\to f(T_0)$. 
Recall from Remark \ref{rmk:cheaptrick} that we are working with convergence in $(D([-1,\infty)),\operatorname{M1})$, and we will embed any function $f\in D([0,\infty))$ into $D([-1,\infty))$ by defining $f(x)=f(0-)$ for $x\in [-1,0)$.

\medskip

For the proof of Theorem \ref{thm:rc}, we will also make use of the following definition and technical lemma from \cite[Definition 5.1 and Lemma 5.4]{cuchiero_propagation_2020}.

\begin{defn}
For $f\in D([-1,\infty))$ define the path functionals
\begin{align}
    \tau_0(f):=\inf\{s\ge 0: f_s\le 0\}\text{ and }\lambda_t(f):=\ind_{\{\tau_0(f)\le t\}}.
\end{align}
\end{defn}

\begin{lem}\label{lem:cross}
Let $(f,f^1,f^2,\dots)\subset D([0,\infty))$. Suppose that $f^n\to f$ in $(D([-1,\infty)),\operatorname{M1})$ and $f$ satisfies the crossing property
\begin{align}\label{eq:cross}
\inf_{0\le s \le h}(f_{\tau_0(f)+s}-f_{\tau_0(f)})<0\text{ for all }h>0.
\end{align}
Then there exists some co-countable set $E$ such that
\begin{align}
\lim_{n\to\infty} \lambda_t(f^n)=\lambda_t(f)
\end{align}
for all $t\in E$.
\end{lem}

\section{Convergence of Minimal Solutions}\label{sec:rc}

We now present the proof of our general convergence result for minimal solutions.

\begin{proof}[Proof of Theorem \ref{thm:rc}]
The condition $F(x)\ge F^n(x)\ge F^m(x)$ for all $x\ge 0$ whenever $m<n$ implies
\begin{align}
\Gamma^{k}[0;X_{0-}^m]_t\le \Gamma^{k}[0;X_{0-}^n]_t \le \Gamma^{k}[0;X_{0-}]_t,\quad t\ge 0
\end{align}
for all $k\ge 1$. Taking the limit as $k\to\infty$ yields $\ubar{\Lambda}^m_t\le \ubar{\Lambda}^n_t \le \ubar{\Lambda}_t$ for all $t\ge 0$. For each $t\ge 0$ we define the pointwise limit
\begin{align}\label{eq:pw}
\Lambda^\infty_t := \lim_{n\to\infty} \ubar{\Lambda}^n_t
\end{align}
and note that $\Lambda^\infty_t\le  \ubar{\Lambda}_t$.

\medskip

The function $t\mapsto \Lambda^\infty_t$ is non-decreasing and hence has at most countably many discontinuities, which we denote by $J$. $\Lambda^\infty$ may not be c\`adl\`ag, so we define $\wh{\Lambda}^\infty_t = \lim_{s\downarrow t}\Lambda^\infty_s$. We have that $\ubar{\Lambda}^n_t \to \wh{\Lambda}^\infty_t$ whenever $t\in J^c$. We extend the domain of all relevant functions to $[-1,\infty)$ by setting them to zero on all of $[-1,0)$. Proposition \ref{prop:M1conv}, gives $\ubar{\Lambda}^n\to \wh{\Lambda}^\infty$ in $(D([-1,t]),\operatorname{M1})$ for all $t\in J^c$ and hence on $(D([-1,\infty)),\operatorname{M1})$.

\medskip

Provided we can show
\begin{align}\label{eq:limit}
\wh\Lambda^\infty_t=\bP (\inf_{0\le s\le t}X_{0-}+B_s-\alpha\wh\Lambda^\infty_s\le 0),
\end{align}
we will have $\wh\Lambda^\infty = \ubar{\Lambda}$ by minimality of the latter. Since $\wh\Lambda^\infty$ is increasing and bounded, \cite[proof of Proposition 3.4(i)]{antunovic_isolated_2011} implies that $t\mapsto X_{0-}+B_t-\wh\Lambda^\infty_t$ satisfies the crossing property \eqref{eq:cross}.
Recall that the pointwise convergence of $F^n\to F$ at continuity points implies that $X^n_{0-}\to X_{0-}$ in law, by the Helly-Bray Theorem \cite[Theorem 11.1.2]{dudley_real_2002}.
Therefore, we may apply Lemma \ref{lem:cross} to obtain
\begin{align}\label{eq:verified}
\lim_{n\to\infty}\bP (\inf_{0\le s\le t}X_{0-}+B_s-\alpha\ubar{\Lambda}^n_s\le 0)=\bP (\inf_{0\le s\le t}X_{0-}+B_s-\alpha\wh\Lambda^\infty_s \le 0)
\end{align}
for $t$ in a co-countable set $E$. The left hand side of the above expression is the definition of $\wh\Lambda^\infty_t$ when $t\in J^c$.
Hence $\wh\Lambda^\infty_t = \ubar{\Lambda}_t$ for $J^c \cap E$ (and in fact everywhere by the right continuity of both) and $\ubar{\Lambda}^n\to \ubar{\Lambda}$ in $(D([-1,\infty)),\operatorname{M1})$. \qedhere

\end{proof}

\begin{rmk}\label{rmk:error}
We note an error in the proof of \cite[Proposition 5.1]{hambly_mckean-vlasov_2022}: the last display equation in that proof gives
\begin{align}
\lim_{\kappa\to \infty}\Lambda^\kappa_t=\lim_{\kappa\to \infty}\bP (X_{0-}+\xi^\kappa +B_t-\alpha\Lambda^\kappa_t\le 0)=\bP (X_{0-}+B_t-\alpha\Lambda^\infty_t\le 0),
\end{align}
which is similar to \eqref{eq:verified}, except that the infimum is missing from the expression for $\Lambda^\kappa_t$. This equality is then verified by the Portmanteau Theorem (see, e.g., \cite[11.1.1]{dudley_real_2002}) since $X_{0-}+B_t-\alpha\Lambda^\infty_t$ has no atoms. However, this fails when the infimum is present. 
\end{rmk}

The proof of Theorem \ref{thm:rc} also gives us some information about left limits: consider a sequence $X^n_{0-}\to X_{0-}$ with $F\le F^n\le F^m$ whenever $m<n$. This holds, for instance, when $X^n_{0-}=X_{0-}-x_n$ where $(x_n)_{n\ge 1}$ is a sequence of non-negative numbers converging monotonically to zero. Then $\ubar{\Lambda}^m_t\ge \ubar{\Lambda}^n_t \ge \ubar{\Lambda}_t$ for all $t\ge 0$ and one may define a pointwise limit $\Lambda^\infty$ analogous to \eqref{eq:pw}. Continuing through the proof we see that the right continuous version $\wh\Lambda^\infty$ solves \eqref{eq:McK-V}. However, without the inequality $\Lambda^\infty_t\le \ubar{\Lambda}_t$ at all $t\ge 0$, we cannot conclude that $\wh\Lambda^\infty$ is a minimal solution. We summarize the considerations of this paragraph as a corollary. 
\begin{coro}\label{coro:ll}
Let $X_{0-}$ have CDF $F$. Let $(X_{0-}^n)_{n\ge 1}$ be random variables on $[0,\infty)$ with CDFs $(F^n)_{n\ge 0}$, respectively, such that $F(x) \le F^n(x)\le F^m(x)$ for all $x\ge 0$ whenever $m<n$, and $F^n(x)\to F(x)$ whenever $x$ is a continuity point of $F$. For $n\ge0$, let $\ubar{\Lambda}^n$ denote the minimal solution to \eqref{eq:McK-V} with initial data $X_{0-}^n$. Then $\Lambda^{\infty}$ defined by
\begin{align}
\Lambda^\infty_t = \lim_{s\downarrow t}\lim_{n\to\infty}\ubar{\Lambda}^n_s
\end{align}
solves \eqref{eq:McK-V} with initial data $X_{0-}$.
\end{coro}

\section{Convergence of Physical Solutions}\label{sec:phys}

In this section, we show that physical solutions converge to solutions which satisfy the physical jump condition. Our proof is by a similar method to \cite[Proof of Theorem 6.4]{cuchiero_propagation_2020}, except that in our case the limit is being approximated by solutions to \eqref{eq:McK-V} with converging initial conditions and free boundaries rather than by finite particle systems. Note also that $\alpha$ appears in a different location in that work.

\begin{proof}[Proof of Theorem \ref{thm:phys}]
Fix $T>0$. For $n \ge 1$ and $t \ge 0$, we let $X_t^n=X_{0-}^n+B_t- \Lambda^n_t$, $\tau^n=\inf \{t\ge 0:X_t^n\le 0\}$, and consider the subprobability measures $\nu_{t-}^n$ defined by 
\begin{align}
\nu_{t-}^n(A)=\bP (\tau^n \ge t, X_t^n\in A).
\end{align}
on $t\in [0,T]$, and $\nu_{t-}^n=0$ on $[-1,0)$.
We define $\nu_{t-}$ analogously for $X_t=X_{0-}+B_t- \Lambda_t$ and $\tau=\inf \{t\ge 0:X_t\le 0\}$, which along with $\Lambda$ give a solution to \eqref{eq:McK-V} by the reasoning in the proof of Theorem \ref{thm:rc}.
 Each of the bounded increasing functions $\Lambda, \Lambda_0, \Lambda_1,\dots$ can have at most countably many discontinuities, and we collect all of these together in a set $J$. Recalling Proposition \ref{prop:M1conv}, we have that $\Lambda^n_t\to\Lambda_t$ for all $t$ in a dense set, which we will call $E$.
 
\medskip

By \cite[Proof of Theorem 6.4, Step 2]{cuchiero_propagation_2020}, for sufficiently small $\veps>0$ for every $n\ge 1$ we have that
\begin{align}
\alpha\nu_{t-}^n([0,z+C\veps^{1/3}])\ge z(1-C\veps^{1/3})
\end{align}
whenever $t,t+\veps\in[0,T]\cap J^c\cap E$ and $z<\Lambda^n_{t+\veps}-\Lambda^n_{t}-C\veps^{1/3}$. The constant $C$ depends only on $\alpha$, $\E|X_{0-}|$, and $T$. 
Using the weak convergence of $X_{0-}^n$ to $X_{0-}$, the $\operatorname{M1}$ convergence of $\Lambda^n$ to $\Lambda$, and following \cite[Proof of Theorem 6.4, Step 1]{cuchiero_propagation_2020} we obtain
\begin{align}\label{eq:subprobn}
\lim_{n\to\infty}\nu_{t-}^n=\nu_{t-}
\end{align}
weakly at all $t\in [-1,T]\cap J^c\cap E$. Note also that 
\begin{align}\label{eq:subprobtime}
\lim_{s\uparrow t}\nu_{s-}^n=\nu_{t-}
\end{align}
by \cite[Proof of Theorem 6.4, first argument of Step 2]{cuchiero_propagation_2020}.

\medskip

Suppose now that $\Lambda_t-\Lambda_{t-}=L>0$ for some $t\in[0,T)$. Then, using the density of $J^c\cap E$ in $[-1,T]$, we may find $(t_m)_{m\ge 1}$ and $(\veps_m)_{m\ge 1}$ with $t_m,t_m+\veps_m\in J^c\cap E$ such that $t_m\uparrow t$, $t_m+\veps_m>t$, and $\veps_m\downarrow 0$. Let $z<L$. Since $\Lambda$ is c\`adl\`ag we must have $z<\Lambda_{t_m+\veps_m}-\Lambda_{t_m}-C\veps_m^{1/3}$ for sufficiently large $m$. Since $t_m,t_m+\veps_m\in J^c\cap E$, then by the definition of $E$, for any $\eta>0$ there is some $N$ such that $\Lambda_{t_m+\veps_m}-\Lambda_{t_m}-C\veps_m^{1/3}<\Lambda^n_{t+\veps_m}-\Lambda^n_{t_m}-C\veps_m^{1/3}+\eta$ for all $n\ge N$. For any $\delta>0$, taking $\eta$ smaller if necessary, using the Portmanteau theorem, along with \eqref{eq:subprobn} and \eqref{eq:subprobtime} gives
\begin{align}
\alpha\nu_{t-}([0, z+\delta])\ge \limsup_{m\to\infty}\limsup_{n\to\infty} \alpha\nu_{t_m}^n([0, z+\delta])\ge \limsup_{m\to\infty} z(1-C\veps_m^{1/3})=z.
\end{align}
Furthermore, by the argument of \cite[Proof of
Proposition 1.2]{hambly_mckean--vlasov_2018}, we must have $L\ge \inf\{x>0:\alpha\bP(\tau\ge t, X_{t-}\in(0,x])<x\}$ since $\Lambda$ is right continuous.  
\end{proof}

\section{Non-uniqueness of Physical Solutions Implies Ill-posedness of Minimal Solutions}\label{sec:nonunique}

In this final section we prove Theorem \ref{thm:nonunique}, hence showing that the minimal solution map at a given initial condition is continuous if and only if the initial condition admits a unique physical solution.

\begin{proof}[Proof of Theorem \ref{thm:nonunique}]
Let $x<0$ and let $\ubar{\Lambda}^x$ be the minimal solution to \eqref{eq:McK-V} with initial condition $X_{0-}+x$.
By Proposition \ref{prop:minisphys}, $\ubar{\Lambda}^x$ is physical. At time zero, the physical jump condition \eqref{eq:phys} yields
\begin{equation}
\begin{split}
\ubar{\Lambda}^x_0-\ubar{\Lambda}^x_{0-}&=\inf\{z>0:\alpha\bP(X_{0-}+x\in(0,z])<z\}\\
&\ge \inf\{z>0:\alpha\bP(X_{0-}\in(0,z])<z\}\\
&=\Lambda_{0}-\Lambda_{0-}.
\end{split}
\end{equation}
Consider 
\begin{align}
t^*:=\inf\{t\ge0:\ubar{\Lambda}^x_t<\Lambda_t\}
\end{align}
and suppose for the sake of contradiction that $t^*<\infty$.
First, we consider the case where 
\begin{align}\label{eq:t*iszero}
t^*=0
\end{align}
and show that this is impossible. By right-continuity of $\Lambda$, for any $\veps>0$ we can find a $\delta>0$ such that $\Lambda_t-\Lambda_0<\veps$ for all $t<\delta$. Assuming \eqref{eq:t*iszero} holds, we may always find some $t_\veps<\delta$ such that
\begin{align}\label{eq:tepsineq}
\ubar{\Lambda}^x_{t_\veps}<\Lambda_{t_\veps}.
\end{align}
On $[0,t_\veps]$, we may upper bound $\Lambda_t-\Lambda_0$ by $\veps$ and lower bound $\ubar{\Lambda}^x_t-\ubar{\Lambda}^x_0$ by zero. Taking $\veps<|x|$ and applying \eqref{eq:McK-V} gives
\begin{equation}
\begin{split}
\ubar{\Lambda}^x_{t_\veps}&\ge \alpha\bP \lb(\inf_{0\le s\le t_\veps} X_{0-}+x+B_s-(\ubar{\Lambda}^x_0-\ubar{\Lambda}^x_{0-})\le 0\rb)\\
&\ge \alpha\bP \lb(\inf_{0\le s\le t_\veps} X_{0-}+B_s-(\Lambda_0-\Lambda_{0-}+\veps)\le 0\rb)\\
&\ge \Lambda_{t_\veps},
\end{split}
\end{equation}
which contradicts \eqref{eq:tepsineq}. Therefore, we may assume that $t^*>0$.

\medskip

The goal now is to show that $\ubar{\Lambda}^x_{t^*}>\Lambda_{t^*}$, which along with the definition of $t^{*}$ will contradict the right-continuity of $\ubar{\Lambda}^x$ and $\Lambda$.
Suppose that $\ubar{\Lambda}^x_{t^*}\le\Lambda_{t^*}$ and let \newline ${\tau=\inf\{t\ge 0: X_{0-}+B_t-\Lambda_t\le 0\}}$. Then by the definition of $\Lambda$ in \eqref{eq:McK-V} we must have
\begin{equation}\label{eq:rcfail}
\begin{split}
\Lambda_{t^*-}+(\ubar{\Lambda}^x_{t^*}-\Lambda_{t^*-})&\le\alpha\bP \lb(\inf_{0\le s< t^*} X_{0-}+B_s-\Lambda_{s}\le 0\rb)\\
 &\quad\quad+\alpha\bP \bigg(\tau\ge t^*, X_{0-}+B_{t^*}\in[\Lambda_{t*-},\ubar{\Lambda}^x_{t^*}]\bigg),
\end{split}
\end{equation}
which holds with equality when $\ubar{\Lambda}^x_{t^*}=\Lambda_{t^*}$, and with possible inequality when $\ubar{\Lambda}^x_{t^*}<\Lambda_{t^*}$ by applying the physical jump condition \eqref{eq:phys} to $\Lambda$ at time $t^*-$. Consider the following augmented path defined on $[0,t^*]$:
\begin{align}
\wt{\Lambda}_t=
\begin{cases}
\Lambda_t, & t\in [0,t^*),\\
\ubar{\Lambda}^x_{t^*}, & t=t^*.
\end{cases}
\end{align}
We see that the right hand side of \eqref{eq:rcfail} is equal to 
\begin{align}
\alpha\bP \lb(\inf_{0\le s\le t^*} X_{0-}+B_s-\wt\Lambda_{s}\le 0\rb),
\end{align}
where we note that the infimum here is taken over $[0,t^*]$ rather than $[0,t^*)$.
Note that, ${\ubar{\Lambda}^x_t\ge\Lambda_t}$ on $[0,t^*)$ by the definition of $t^*$. Therefore, $\ubar{\Lambda}^x_t\ge\wt{\Lambda}_t$ on $[0,t^*]$. Since $|x|>0$ and $t^*>0$, the following holds with strict inequality:
\begin{align}
\alpha\bP \lb(\inf_{0\le s\le t^*} X_{0-}+B_s-\wt\Lambda_{s}\le 0\rb) < \alpha\bP \lb(\inf_{0\le s\le t^*} X_{0-}+x+B_s-\ubar{\Lambda}^x_{s}\le 0\rb)=\ubar{\Lambda}^x_{t^*}.
\end{align}
Putting this together with \eqref{eq:rcfail} yields $\ubar{\Lambda}^x_{t^*}<\ubar{\Lambda}^x_{t^*}$. This contradiction implies that ${\ubar{\Lambda}^x_{t^*}>\Lambda_{t^*}}$, which as mentioned before, contradicts the right-continuity of $\ubar{\Lambda}^x$ and $\Lambda$ at $t^*$. Hence ${t^*=\infty}$ and we must have $\ubar{\Lambda}^x_t\ge\Lambda_t$ for all $t\ge0$. As $x$ was arbitrary, we may take the limit as $x$ goes to 0, and the inequality $\Lambda^0_t\ge\Lambda_t$ holds for all $t\ge 0$.

\medskip

Suppose now that $X_{0-}$ admits two non-equal physical solutions $\Lambda$ and $\hat{\Lambda}$. Suppose there exists some $t>0$ such that $\Lambda_t>\hat{\Lambda}_t$ (if this is not the case then we can swap the labelling on $\Lambda$ and $\hat{\Lambda}$ so that this holds). Then by the above arguments, ${\Lambda^0_t\ge\Lambda_t>\hat{\Lambda}_t\ge \ubar{\Lambda}_t}$. This shows that $\Lambda^0\neq \ubar{\Lambda}$. \qedhere

\end{proof}

\section*{Acknowledgements}

The author was partially supported by an NSERC PGS-D scholarship and a Princeton SEAS innovation research grant. The author extends gratitude to his PhD advisor, Misha Shkolnikov; as well as Hezekiah Grayer II, Scander Mustapha, and Stefan Rigger for many useful conversations.

\printbibliography

\bigskip\bigskip

\end{document}